\documentclass[11pt,a4paper]{article}

\usepackage[utf8]{inputenc} 
\usepackage[T1]{fontenc}

\usepackage{amsmath}
\usepackage{amsthm}
\usepackage{amssymb}
\usepackage{amsfonts}
\usepackage{dsfont}
\usepackage{url}
\usepackage{graphicx}
\usepackage{hyperref}

\usepackage{tikz}
\usepackage{float}
\usepackage{pgfplots}
\pgfplotsset{compat=1.5}

\renewcommand{\P}{\mathbb{P}}
\newcommand{\E}{\mathbb{E}}
\newcommand{\N}{\mathbb{N}}
\newcommand{\Z}{\mathbb{Z}}
\newcommand{\vmax}{v_{\max}}
\newcommand{\vmin}{v_{\min}}
\newcommand{\limn}{\lim_{n \to \infty}}

\theoremstyle{plain}
\newtheorem{thm}{Theorem}[section]

\newtheorem{lemma}[thm]{Lemma}
\newtheorem{con}[thm]{Conjecture}

\newcommand{\Addresses}{{%
  \bigskip
  \footnotesize

  T.~Höfelsauer, \textsc{Fakultät für Mathematik, Technische Universität München,
    \newline Boltzmannstr. 3, 85748 Garching, Germany}\par\nopagebreak
  \textit{E-mail address:} \texttt{thomas.hoefelsauer@tum.de}

  \medskip

  F.~Weidner, \textsc{Fakultät für Mathematik, Technische Universität München, \newline
    Boltzmannstr. 3, 85748 Garching, Germany}\par\nopagebreak
  \textit{E-mail address:} \texttt{felizitas.weidner@tum.de}

}}

\begin{document}
\title{The speed of frogs with drift on $\Z$}
\author{Thomas Höfelsauer and Felizitas Weidner}
\maketitle


\begin{abstract}
In this article we consider the frog model with drift on $\Z$ and investigate the behaviour of the cloud of the frogs. In particular, we show that the speed of the minimum equals the speed of a single frog and prove some properties of the speed of the maximum. In addition, we show a limit theorem for the empirical distribution.


\end{abstract}


\section{Description of the model and results}
We consider a system of interacting random walks, known as the frog model, on the one-dimensional lattice $\Z$. In this model the particles performing the random walks are thought of as frogs. Initially, there is one active frog at the origin and one sleeping frog at every other site. The active frog performs a discrete time simple random walk jumping to the right with probability~$p$ and to the left with probability $1-p$ where $p \geq \frac12$. Whenever a sleeping frog is visited by an active frog, it is activated and starts a simple random walk itself with the same drift parameter $p$, independently of all other active frogs. It can wake up other sleeping frogs as well.

Gantert and Schmidt prove in \cite{GS09} that the frog model as described here is transient almost surely for every $p > \frac12$. This means that the origin is visited only finitely many times by active frogs. Indeed, they consider a more general model with a random initial configuration of sleeping frogs and give criteria for recurrence and transience.
In \cite{AMP02} and \cite{AMPR01} Alves et al.~prove a shape theorem in $\Z^d$ when the underlying random walk is symmetric. Transience and recurrence for frogs on $\Z^d$ is discussed in \cite{TW99}, \cite{P01} and \cite{DP14}, for frogs on trees in \cite{HJJ14} and \cite{HJJ15}, and in a more general setting in \cite{KZ15}. Other aspects are discussed in \cite{AMP02}, \cite{LMP05} and \cite{GNR15}. An overview and a collection of problems up to the year 2003 can be found in~\cite{P03}.

\enlargethispage{\baselineskip}

In this article we investigate how the set of sites occupied by active frogs evolves over time. In particular, we consider the left and right front of this set and study their speed.
In continuous time the right front is studied in \cite{CQR09} and \cite{BR10}.

Let us introduce the model in a more formal way: Let $(X_k^i)_{i\in \Z, k\in\N}$ be an array of independent and identically distributed random variables such that $\P(X_1^0=1)=p=1-\P(X_1^0=-1)$. For $i \in \Z, n \in \N_0$ define $S_n^i = \sum_{k=1}^n X_k^i$. The sequence $(S_n^i)_{n \in \N_0}$ describes the trajectory of the frog initially at site~$i$. It starts to follow that trajectory once it is woken up, which is a random time depending on the behaviour of the frogs activated earlier. We denote this activation time of frog $i$ by $T_i$. A formal definition can be found in \cite{AMP02}. The location $Z_n^i$ of the frog initially at site $i \in \Z$ at time $n \in \N_0$ is given by
\begin{equation*}
Z_n^i = 
\begin{cases}
i		& \text{for $n<T_i$},\\
i+S_{n-T_i}^i	& \text{for $n\geq T_i$}. 
\end{cases}
\end{equation*}
Let $A_n$ denote the set of active frogs at time $n$, i.e.~$A_n=\{i \in \Z \colon T_i \leq n\}$. Further, we define $M_n=\max_{i \in A_n} Z_n^i$ and $m_n=\min_{i \in A_n} Z_n^i$. Thus, $M_n$ describes the maximum and $m_n$ the minimum of the locations of the active frogs. We refer to $M_n$ and $m_n$ as the maximum and the minimum.
One can show that there are non-zero constants $\vmax$ and $\vmin$ such that
\begin{align*}
\vmax &= \lim_{n\to \infty} \frac{M_n}{n} \quad \text{a.s.}\\
\vmin &= \lim_{n\to \infty} \frac{m_n}{n} \quad \text{a.s.}
\end{align*}
The existence of $\vmax$ is well known and is discussed in Section~\ref{chapter_proofs}, and the existence of $\vmin$ is part of Theorem~\ref{theorem_speed} below.
We call $\vmax$ the speed of the maximum and $\vmin$ the speed of the minimum. 
We study $\vmax$ and $\vmin$ as functions of the drift parameter $p$. 
First, we show that the speed of the minimum equals the speed of a single frog.

\begin{thm} \label{theorem_speed}
For $p > \frac12$ the speed of the minimum exists and is given by
\begin{equation*}
\vmin = 2 p -1.
\end{equation*}
\end{thm}

In the following two theorems we discuss some properties of the speed of the maximum.

\begin{thm} \label{theorem_monotonicity}
The speed of the maximum is an increasing function in $p$.
\end{thm}

\begin{thm} \label{theorem_vmax_smaller_one}
For $p < 1$ it holds that $\vmax < 1$.
\end{thm}

In comparison to the last result note that for branching random walk on $\Z$ with binary branching the speed of the maximum equals $1$ for every $p \geq \frac12$, see e.g.~\cite[Theorem 18.3]{P99}.

In addition to studying the behaviour of the minimum and the maximum we investigate the distribution of the active frogs.
In the limit they are distributed uniformly inbetween the minimum and the maximum. To make this statement precise we rescale the positions of all active frogs at time~$n$ roughly to the interval $[0,1]$ and then consider the empirical distribution $\mu_n$, which is defined for $p < 1$ by
\begin{equation*}
\mu_n(B) = \frac{1}{|A_n|} \sum_{i \in A_n} \mathds{1}_{\bigl\{\frac{Z_n^i-\vmin n}{(\vmax-\vmin)n}\in B\bigr\}}
\end{equation*}
for every Borel set $B \subseteq [0,1]$. Note that $\mu_n$ is a random measure.

\begin{thm} \label{theorem_equidistribution}
Almost surely, the empirical distribution $\mu_n$ converges weakly to the Lebesgue measure~$\lambda$ on $[0,1]$ as $n \to \infty$.
\end{thm}


\section{Proofs}\label{chapter_proofs}


The existence of the speed of the maximum is proved using Liggett's Subadditive Ergodic Theorem. Indeed, this theorem yields more information which we use throughout this article. We summarize it in the following Lemma.

\begin{lemma} \label{lemma_speed_max}
For each $p$ there exists a positive constant $\vmax$ such that
\begin{equation*}
  \vmax = \lim_{n\to \infty} \frac{M_n}{n} \quad \text{a.s.}
 \end{equation*}
Furthermore, 
 \begin{equation} \label{L1_convergence}
  \vmax^{-1} = \lim_{i\to \infty} \frac{T_i}{i} = \lim_{i\to \infty} \frac{\E[T_i]}{i} = \inf_{i \in \N} \frac{\E[T_i]}{i} \quad \text{a.s.}
 \end{equation}
\end{lemma}

\begin{proof}
Let $T_{i,j}$ denote the activation time of the frog at site $j$ when initially there is one active frog at site $i$ and one sleeping frog at every other site.
An application of Liggett's Subadditive Ergodic Theorem (see e.g.~\cite{L85}) to the times $(T_{i,j})_{i,j \in \Z}$ shows the existence of a positive constant $\vmax$ such that~\eqref{L1_convergence} holds. For $p=\frac12$ this is proved for a more general model by Alves et al.~in~\cite{AMP02}. In our setting their argument immediately applies to $p >\frac12$ as well.

By a standard argument it now follows that $\limn \frac{M_n}{n}$ exists almost surely:
There exists a unique random sequence $(k_n)_{n\in\N}$ with values in $\N_0$ such that $T_{k_n} \leq n < T_{k_n+1}$. Note that $\limn k_n=\infty$. Hence,
\begin{equation*}
 \limn \frac{T_n}{n}= \limn \frac{T_{k_n}}{k_n} = \limn\frac{n}{k_n} \quad \text{a.s.}
\end{equation*}
Obviously, $k_n - (n-T_{k_n}) \leq M_n \leq k_n$. This implies
\begin{equation*}
 \frac{k_n}{n} - \biggl( 1- \frac{T_{k_n}}{k_n} \cdot \frac{k_n}{n} \biggr) \leq \frac{M_n}{n} \leq \frac{k_n}{n}.
\end{equation*}
Taking limits yields the claim.
\end{proof}


In order to prove Theorem~\ref{theorem_speed} we compare the frogs initially on non-negative sites with independent random walks. The speed of the minimum of independent random walks can be computed explicitly which is done in the first of the following lemmas. Then it remains to deal with the frogs initially on negative sites. Luckily, they can be ignored due to the transience of the frog model proved in \cite{GS09} by Gantert and Schmidt.

We often need to talk about the frogs initially on negative sites. To keep the sentence structure simple we refer to them as the negative frogs.
Analogously we speak of non-negative and positive frogs. Also for any $k\in \Z$ the frog initially on site $k$ is called frog $k$.

\begin{lemma} \label{lemma_indrw}
It holds that
\begin{equation*}
 \limn \frac{1}{n} \min_{i \in \{-n, \ldots, n\}} S_n^i = 2p-1 \quad \text{a.s.}
\end{equation*}
\end{lemma}

\begin{proof}
We only need to prove $\liminf_{n \to \infty} \frac{1}{n} \min_{i \in \{-n, \ldots, n\}} S_n^i \geq 2p-1$. For all $\varepsilon >0$ we have
\begin{eqnarray*}
\P\biggl( \frac1n \min_{i\in \{-n, \ldots, n\}}  S_n^i \leq 2p-1 - \varepsilon \biggr)
    &=    &\P\biggl( \bigcup_{i = -n}^n \Bigl\{ \frac{S_n^i}{n} \leq 2p-1 - \varepsilon \Bigr\} \biggr) \\
    &\leq & (2n+1) \P\biggl(\frac{S_n^0}{n} \leq 2p-1 - \varepsilon \biggr).
\end{eqnarray*}
By Cramér's Theorem the probability in the last term of this calculation decays exponentially fast in $n$, hence, it is summable. An application of the Borel-Cantelli Lemma and letting $\varepsilon \to 0$ completes the proof.
\end{proof}

This result now enables us to prove a formula for the speed of the minimum of the non-negative frogs.

\begin{lemma} \label{lemma_coupling}
Let $A_n^+= \{i \geq 0 \colon T_i \leq n\}$. Then
\begin{equation*}
\limn \frac1n \min_{i \in A_n^+} Z_n^i = 2p -1 \quad \text{a.s.}
\end{equation*}
\end{lemma}

Before proving Lemma~\ref{lemma_coupling} we make another short observation. Obviously $\vmax$ is at least as big as the speed of a single frog, i.e.~$\vmax \geq 2p-1$. In fact, this inequality is strict for all $p \geq \frac12$. For $p = \frac12$ this fact is known from~\cite{AMP02}.

\begin{lemma} \label{lemma_vmaxgreaterdrift}
For $\frac12 < p < 1$ it holds that $\vmax > 2p-1$.
\end{lemma}

\begin{proof}
The key point in this proof is to notice that $\E[T_1]< \E[T_1^s]$ holds, where $T_1^s=\inf\{n \in \N \colon S_n^0=1\}$ denotes the hitting time of the point $1$ of a single simple random walk with drift $2p-1$.
Hence, by Lemma~\ref{lemma_speed_max}
\begin{equation*}
\vmax^{-1} =    \inf_{i \in \N} \frac{\E[T_{i}]}{i}
	   \leq \E[T_1]
	   <    \E[T_1^s]
	   =    \frac{1}{2p-1}. \qedhere
\end{equation*}
\end{proof}

One can of course find better lower bounds for the speed of the maximum by estimating $\E[T_i]$ for $i \geq 1$, but this is not done in this article.

\begin{proof}[Proof of Lemma~\ref{lemma_coupling}]
It is enough to show $\liminf_{n\to\infty} \frac1n \min_{i \in A_n^+} Z_n^i \geq 2p-1$ almost surely.
 
In this proof we use a different but equivalent way of defining the movement of the frogs. For every $i \in \Z$ define
\begin{equation*}
 \widetilde{Z}_n^i =
         \begin{cases}
          i                   & \text{for $n < T_i$,}  \cr
          i+S_{n}^i-S_{T_i}^i & \text{for $n\geq T_i$.}
         \end{cases}
\end{equation*}
Note that $(\widetilde{Z}_n^i)$ equals $(Z_n^i)$ in distribution. We now want to compare the trajectory $(\widetilde{Z}_n^i)_{n \in \N_0}$ of each frog with the trajectory $(S_n^i)_{n\in\N_0}$ of the corresponding simple random walk. From time $T_i$ onwards they move synchronously by the above definition. Therefore, we only need to compare their locations at time $T_i$. Note that $\widetilde{Z}_{T_i}^i=i$ and define $G=\{i\geq 0\colon S_{T_i}^i \leq i\}$ to be the set of good frogs. Now $i\in A_n^+\cap G$ implies $S_n^i \leq \widetilde{Z}_n^i$ for all $n \in \N$, i.e.~all good frogs stay to the right of their corresponding random walk. Hence,
\begin{align} \label{proof_speed_1}
\min_{i\in A_n^+} \widetilde{Z}_n^i &\geq \min_{i\in A_n^+} S_n^i - \sum_{i\in G^c\cap A_n^+} \bigl(S_n^i- \widetilde{Z}_n^i \bigr) \nonumber \\ 
                                    &\geq \min_{i\in A_n^+} S_n^i - \sum_{i\in G^c} \bigl(S_{T_i}^i-i \bigr).
\end{align}

We claim that the set $G^c$ is finite almost surely. For $p=1$ this is obviously true. For $p < 1$ it is enough to show that
\begin{equation}
\lim_{i \to \infty} \frac{S_{T_i}^i- i}{T_i} = 2p-1-\vmax \quad \text{a.s.} \label{proof_speed_2}
\end{equation}
since by Lemma~\ref{lemma_vmaxgreaterdrift} the last term is strictly negative and hence $S_{T_i}^i- i >0$ can occur only for finitely many $i \geq 0$ almost surely. 

Note that $(S_n^i)_{n \le T_i}$ is independent of the movement of the frogs up to time~$T_i$. Therefore, $S_{T_i}^i$ equals $S_{T_i}^0$ in distribution. 
Using a standard large deviation estimate we get for every $\varepsilon > 0$
\begin{equation*}
\P\biggl(\frac{S_{T_i}^i}{T_i} \leq 2p-1-\varepsilon \biggr) =    \P\biggl(\frac{S_{T_i}^0}{T_i} \leq 2p-1-\varepsilon \biggr)
                                                             \leq \E \bigl[ \mathrm{e}^{-cT_i} \bigr] 
                                                             \leq \mathrm{e}^{-ci} 
\end{equation*}
where $c=c(\varepsilon,p)>0$ is a constant. By symmetry also $\P\Bigl(\frac{S_{T_i}^i}{T_i} \geq 2p-1+\varepsilon \Bigr)$ decays exponentially fast in $i$. An application of the Borel-Cantelli Lemma and letting $\varepsilon \to 0$ thus shows
\begin{equation*}
\limn \frac{S_{T_i}^{i}}{T_i} = 2p-1 \quad \text{a.s.}
\end{equation*}
Further, we know from Lemma~\ref{lemma_speed_max} that $\lim_{i \to \infty}\frac{i}{T_i} = \vmax$ almost surely. This proves equation~\eqref{proof_speed_2} which implies that $G^c$ is finite almost surely. 

Therefore, the second term on the right side in inequality~\eqref{proof_speed_1} is finite almost surely. Also note that it does not depend on $n$. Thus,
\begin{equation*}
\liminf_{n\to \infty} \frac1n \min_{i\in A_n^+} \widetilde{Z}_n^i  \geq \liminf_{n\to \infty} \frac1n \min_{i\in A_n^+} S_n^i \quad \text{a.s.}  
\end{equation*}
As $A_n^+ \subseteq \{-n, \ldots, n\}$ an application of Lemma~\ref{lemma_indrw} finishes the proof.
\end{proof}

\begin{proof}[Proof of Theorem \ref{theorem_speed}]
As shown in \cite[Theorem 2.3]{GS09} the frog model with drift, as considered here for $p \neq \frac12$, is transient almost surely. This means that the origin is visited by only finitely many frogs almost surely. Therefore only finitely many negative frogs are ever activated. Hence, Theorem~\ref{theorem_speed} follows from Lemma~\ref{lemma_coupling}.
\end{proof}


Next we prove that the speed of the maximum is an increasing function in the drift parameter $p$. Though this statement might at first seem obvious, no direct coupling of the frog models for different drift parameters seems possible, since for smaller values of $p$ more negative frogs will eventually be woken up, which might help in pushing the front forward. But we can ignore all these frogs without changing the speed of the maximum, similar to the proof of Theorem~\ref{theorem_speed}. This is shown in the next lemma. We therefore consider the frog model without negative frogs. It evolves in the same way as our usual frog model, but has another initial configuration. Here we assume that there is one sleeping frog at every positive integer, one active frog at $0$ and no frogs on negative sites. We denote the activation time of the $i$-th frog in the frog model without negative frogs by $T^+_i$.

\begin{lemma}\label{lemma_negative_frogs}
It holds that
\begin{equation*}
\vmax^{-1}= \lim_{n\to\infty} \frac{T^+_n}{n} \quad \text{a.s.}
\end{equation*}
\end{lemma}

\begin{proof}
We only need to prove $\limsup_{n \to \infty} \frac{T^+_n}{n} \leq \vmax^{-1} $ almost surely.
First, we show that the speed of the maximum of all negative frogs in the usual frog model equals $2p-1$ almost surely, i.e.~setting $A_n^- =\{i <0\colon T_i \leq n\}$ we prove that
\begin{equation} \label{proof_lemma_negative_frogs_1}
 \limn \frac1n \max_{i \in A_n^-} Z_n^i = 2p -1 \quad \text{a.s.}
\end{equation}

For $p > \frac12$ only finitely many negative frogs will ever be activated almost surely as remarked in the proof of Theorem~\ref{theorem_speed} and proved in \cite{GS09}. In this case equation~\eqref{proof_lemma_negative_frogs_1} is thus obvious.

If $p=\frac12$, then by symmetry the claim follows from Lemma~\ref{lemma_coupling}.

Let $E$ be the set of all positive frogs which are activated by negative frogs, meaning that at the time of their activation at least one negative frog is present. Since $\vmax > 2p-1$ as proved in Lemma~\ref{lemma_vmaxgreaterdrift} and by equation~\eqref{proof_lemma_negative_frogs_1} the set $E$ is finite almost surely.

Hence, $T = \sup_{i \in E}(T^+_i -T_i)$ is an almost surely finite random variable. 
For all $i \in E$ we thus have $T^+_i \leq T_i + T$. Actually, this inequality is true for all $i \in \N_0$, which immediately implies the claim of the lemma. 

The inequality can e.g.~be proven inductively. For $i=0$ the inequality is obviously true as $T_0^+ = T_0=0$. Now assume that $i \in \N$ and $T_j^+ \leq T_j+T$ holds for all $0 \leq j \leq i-1$. If $i \in E$, there is nothing to show. Otherwise, let $0 \leq k \leq i-1$ be the (random) frog that activates the frog $i$ in the normal version of the model. Then we have 
\begin{equation*}
 T_i^+ \leq T_{k}^+ + (T_i-T_{k}) \leq T_i + T.
\end{equation*}
Note here that in both models the frogs follow the same paths, they might just be activated at different times.
\end{proof}

\begin{proof}[Proof of Theorem~\ref{theorem_monotonicity}]
Using a standard coupling of the random variables $(X_k^i)_{i\in \Z,k\in\N}$ we can achieve that $T^+_i(p)$ is monotone decreasing in $p$. As $\vmax(p) = \lim_{n\to\infty} \frac{n}{T^+_n}$ almost surely by Lemma~\ref{lemma_negative_frogs} we conclude that $\vmax(p)$ is increasing in $p$.
\end{proof}

In order to bound the speed of the maximum from above we prove an upper bound for the number of frogs in the maximum. We do this for a slightly modified frog model:
Each time the maximum moves to the left we put a sleeping frog at the site that has just been left by the maximum. Hence, in this new model there is one sleeping frog at every site to the right of the maximum at any time. Further notice that, except at time $0$, there are always at least two frogs in the maximum. We use the same notation as in the usual frog model, but add an index ``$\text{mod}$'' when referring to the modified model.
Further, let $a_n$ denote the number of frogs in the maximum in the modified frog model.

\begin{lemma} \label{lemma_uniform_bound}
For $\frac12 < p < 1$ and all $n \in \N$ it holds that
\begin{equation*}
 \E[a_n] \leq \frac{(2-p)p}{(1-p)(2p-1)}.
\end{equation*}
\end{lemma}

\begin{proof}
We prove bounds not only for the number of frogs in the maximum, but for every other site as well. Therefore, let $a_n(k)$ be the number of frogs at location $M_n^{\text{mod}}-2k$ for $k,n \in \N_0$. 
We prove by induction over $n$ that for all $n,k \in \N_0$
\begin{equation}\label{proof_uniform_bound_1}
 \E[a_n(k)] \leq \frac{(2-p)p}{(1-p)(2p-1)p^k}.
\end{equation}
For $n=0$ and $n=1$ one easily checks that the claim is true. Assume that the claim holds for some integer $n \in \N$.

First we show inequality~\eqref{proof_uniform_bound_1} for $k=0$. Distinguishing whether all $a_n$ particles in the maximum at time $n$ move to the left or not in the next step one calculates
\begin{align*}
 \E[a_{n+1}] &= \E\bigl[(1-p)^{a_{n}} \bigl(a_n+p a_{n}(1)\bigr) \bigr]\\
             & \quad + \E \Bigl[\bigl(1-(1-p)^{a_n}\bigr) \Bigl( \frac{p a_n}{1 -(1-p)^{a_n}} + 1 \Bigr) \Bigr]\\
             &= \E\bigl[(1-p)^{a_{n}} \bigl(a_n+p a_{n}(1) -1\bigr) +p a_n + 1 \bigr].
\end{align*}             
Note here that the expectation of a binomial random variable with para\-meters $p > 0$ and $k \in \N$ conditioned on being at least $1$ is given by $\frac{p k}{1 -(1-p)^{k}}$. 
Using $a_n \geq 2$ yields
\begin{equation} \label{proof_uniform_bound_2}
 \E[a_{n+1}] \leq (1-p)^{2}\E\bigl[a_n+p a_{n}(1) -1 \bigr] +p \E[a_n] + 1.           
\end{equation}

Inserting the induction hypothesis \eqref{proof_uniform_bound_1} in \eqref{proof_uniform_bound_2} the claim follows after a straightforward calculation.

For $k = 1$  an analogous calculation yields
\begin{align}
 \E[a_{n+1}(1)] &= \E\bigl[(1-p)^{a_{n}} \bigl(pa_n(2) + (1-p) a_{n}(1)\bigr) \bigr] \nonumber\\
                & \quad + \E\Bigl[\bigl(1-(1-p)^{a_n} \bigr) \Bigl(a_n - \frac{pa_n}{1-(1-p)^{a_n}} +p a_n(1) \Bigr) \Bigr] \nonumber \\
                &=  \E\bigl[(1-p)^{a_{n}} \bigl(pa_n(2) - (2p-1) a_{n}(1) -a_n\bigr) \bigr] \nonumber\\
                & \quad + \E\bigl[(1-p)a_n + pa_n(1) \bigr]. \label{proof_uniform_bound_3}
\intertext{For $k \geq 2$ one gets}
 \E[a_{n+1}(k)] &= \E\bigl[(1-p)^{a_{n}} \bigl(pa_n(k+1) + (1-p) a_{n}(k)\bigr) \bigr] \nonumber\\
                & \quad + \E\bigl[(1-(1-p)^{a_n}) \bigl(p a_n(k) + (1-p)a_n(k-1) \bigr) \bigr] \nonumber \\
                &=  \E\bigl[(1-p)^{a_{n}} \bigl(pa_n(k+1) - (2p-1) a_{n}(k) -(1-p)a_n(k-1)\bigr) \bigr] \nonumber\\
                & \quad + \E\bigl[(1-p)a_n(k-1) + pa_n(k) \bigr]. \label{proof_uniform_bound_4}
\intertext{Thus, for $k \geq 1$ equations~\eqref{proof_uniform_bound_3} and~\eqref{proof_uniform_bound_4} imply}
 \E[a_{n+1}(k)] &\leq p(1-p)^{2} \E[a_n(k+1)] + p \E[a_n(k)] + (1-p)\E[a_n(k-1)]. \label{proof_uniform_bound_5}
\end{align}
As before, inserting the induction hypothesis \eqref{proof_uniform_bound_1} into inequality~\eqref{proof_uniform_bound_5} completes the proof.
\end{proof}


\begin{proof}[Proof of Theorem~\ref{theorem_vmax_smaller_one}]
Consider the event that in the modified frog model at time $n$ all $a_n$ frogs sitting in the maximum move to the left. Using Jensen's inequality and Lemma~\ref{lemma_uniform_bound}, we conclude that the probability of this event is bounded from below by
\begin{equation*}
 \E\bigl[(1-p)^{a_n}\bigr] \geq (1-p)^{\E[a_n]} \geq (1-p)^{\frac{(2-p)p}{(1-p)(2p-1)}}.
\end{equation*}
Therefore, for all $n \in \N$
\begin{align*}
  \E\bigl[T_{n+1}^{\text{mod}}-T_n^{\text{mod}}\bigr] & \geq  1+ 2 \, \E\bigl[(1-p)^{a_n}\bigr]\\
                                                      & \geq  1 + 2 (1-p)^{\frac{(2-p)p}{(1-p)(2p-1)}}.
\end{align*}
Clearly, in the modified model, frogs are activated no later than in the normal version of the frog model. Thus, 
\begin{equation*}
 \E[T_n] \geq \E\bigl[T_n^{\text{mod}}\bigr] = \sum_{k=1}^n \E\bigl[T_k^{\text{mod}} - T_{k-1}^{\text{mod}}\bigr] \geq \Bigl( 1 + 2 (1-p)^{\frac{(2-p)p}{(1-p)(2p-1)}}\Bigr) n.
\end{equation*}
By Lemma~\ref{lemma_speed_max} we conclude
\begin{equation*}
 \vmax^{-1} = \inf_{n \in \N} \frac{\E[T_n]}{n} \geq  1 + 2 (1-p)^{\frac{(2-p)p}{(1-p)(2p-1)}} >1. \qedhere
\end{equation*}

\end{proof}


It remains to prove Theorem~\ref{theorem_equidistribution}. The idea of the proof is quite simple: From the point of view of the minimum the front moves with a positive speed, but all the frogs only fluctuate around their locations with $\sqrt{n}$, so basically they stay where they are.

First, we show that for large enough times $n$ all active frogs do not deviate much from their expected locations. To formalize this statement we define the set $G_n = \{i \in A_n\colon \bigl|Z_n^i - \E [Z_n^i]\bigr| < n^{3/4}\}$.

\begin{lemma} \label{lemma_equidistribution_1}
Almost surely, $G_n=A_n$ for all $n$ large enough.
\end{lemma}

\begin{proof}
As $A_n \subseteq \{-n, \ldots, n\}$ we have
\begin{align}\label{proof_uni_distr_1}
\P \bigl( A_n \not = G_n \bigr)
			&=    \P \Bigl(\bigcup_{i \in A_n} \bigl\{\bigl|Z_n^i-\E [Z_n^i]\bigr| \geq n^{3/4}\bigr\}\Bigr) \nonumber \\
			&\leq \sum_{i=-n}^n \P\bigl(\bigl|Z_n^i-\E[Z_n^i]\bigr| \geq n^{3/4}\bigr) \nonumber \\
			&=    \sum_{i=-n}^n \sum_{k=0}^n \P\bigl(\bigl|Z_n^i-\E[Z_n^i]\bigr| \geq n^{3/4} \bigl| T_i=k\bigr) \bigr. \cdot \P(T_i=k).
\end{align}
Further, for every $i \in \Z$ and $0 \leq k \leq n$ it holds that
\begin{align*}
\P\bigl(\bigl|Z_n^i-\E[Z_n^i]\bigr| \geq n^{3/4} \bigl| T_i=k\bigr) \bigr.
			&=    \P\bigl(\bigl|S_{n-k}^i-\E[S_{n-k}^i]\bigr| \geq n^{3/4} \bigr) \\
			&\leq 2 \operatorname{exp}\Bigl(-\frac{n^{3/2}}{4(n-k)} \Bigr) \\
			&\leq 2 \operatorname{exp}\Bigl(-\frac{n^{1/2}}{4}  \Bigr).
\end{align*}
In the first inequality in the above estimate we use H{\"o}ffding's inequality.
Thus, \eqref{proof_uni_distr_1} implies
\begin{equation*}
\P \bigl( A_n \not = G_n \bigr) \leq 2 \operatorname{exp}\Bigl(-\frac{n^{1/2}}{4}  \Bigr) \sum_{i=-n}^n \sum_{k=0}^n \P(T_i=k) \leq 2 (2n+1) \operatorname{exp}\Bigl(-\frac{n^{1/2}}{4}  \Bigr)
\end{equation*}
which is summable. An application of the Borel-Cantelli Lemma completes the proof.
\end{proof}

For $\varepsilon >0$ and $x \in [0,1]$ define
\begin{align*}
 L_n(x, \varepsilon) &= 
 \begin{cases}
  \bigl\{i\in \Z\colon -(\vmax-\varepsilon)n \leq i \leq \bigl((2x-1) \vmax - \varepsilon \bigr)n \bigr\} & \text{for $p=\frac12$}, \\
  \bigl\{i\in \Z\colon 0 \leq i \leq (x \vmax - \varepsilon)n \bigr\} & \text{for $p>\frac12$} \\
 \end{cases}\\
 \intertext{and}
 R_n(x,\varepsilon) &= 
 \begin{cases}
  \bigl\{i\in \Z\colon \bigl((2x-1) \vmax + \varepsilon\bigr)n \leq i \leq (\vmax-\varepsilon)n \bigr\} & \text{for $p=\frac12$}, \\
  \bigl\{i\in \Z\colon (x \vmax + \varepsilon)n \leq i \leq  (\vmax-\varepsilon)n \bigr\} & \text{for $p>\frac12$}.
 \end{cases}
\end{align*}

\begin{lemma} \label{lemma_equidistribution_2}
For $n$ large enough, $i \in L_n(x,\varepsilon) \cap G_n$ implies
\begin{equation}\label{lemma_equidistribution_21}
\frac{Z_n^i-\vmin n}{(\vmax-\vmin)n} \leq x,
\end{equation}
whereas $i \in R_n(x,\varepsilon) \cap G_n$ implies
\begin{equation}\label{lemma_equidistribution_22}
\frac{Z_n^i-\vmin n}{(\vmax-\vmin)n} \geq x.
\end{equation} 
\end{lemma}

\begin{proof}
For $p=\frac12$ note that by symmetry $\vmin={-}\vmax$. Thus, \eqref{lemma_equidistribution_21} holds if and only if $Z_n^i \leq (2x-1)\vmax n$. 
Assume $i \in L_n(x,\varepsilon)\cap G_n$. A straightforward calculation shows
\begin{equation*}
 Z_n^i  \leq \E[Z_n^i] + n^{3/4} 
        =    i + n^{3/4} 
        \leq (2x-1)\vmax n 
\end{equation*}
for $n$ big enough.
Analogously, one shows \eqref{lemma_equidistribution_22} in this case.

For $p>\frac12$ the proof works essentially in the same way as for $p=\frac12$, but the estimation of $\E[Z_n^i]$ is less trivial. We have $\E[Z_n^i] =  i + (n-\E[T_i]) \vmin.$
For $i \in L_n(x,\varepsilon) \cap G_n$ we thus get
\begin{equation*}
 Z_n^i  \leq \E[Z_n^i] + n^{3/4} 
        =    \vmin n + \frac{i}{\vmax} \Bigl( \vmax - \frac{\E[T_i]}{i} \vmin \vmax \Bigr) + n^{3/4}. 
\end{equation*}

Lemma~\ref{lemma_speed_max} yields that $\frac{\E[T_i]}{i} \geq \inf_{i\in \N} \frac{\E[T_i]}{i} = \vmax^{-1}$. 
Hence, for $n$ big enough
\begin{align*}
 Z_n^i & \leq \vmin n + \frac{i}{\vmax} ( \vmax - \vmin) + n^{3/4} \\
       & \leq \vmin n + x ( \vmax - \vmin) n,
\end{align*}
as claimed in \eqref{lemma_equidistribution_21}. On the other hand, $i \in R_n(x,\varepsilon) \cap G_n$ analogously implies
\begin{equation*}
 Z_n^i \geq  \vmin n + \frac{i}{\vmax} \Bigl( \vmax - \frac{\E[T_i]}{i} \vmin \vmax \Bigr) - n^{3/4}. 
\end{equation*}
Since $\lim_{i \to \infty} \frac{\E[T_i]}{i}=\vmax^{-1}$ and $i$ tends to infinity whenever $n$ does by the definition of $R_n(x,\varepsilon)$, we know that $\frac{\E[T_i]}{i} \leq \vmax^{-1}+\delta \varepsilon$ for $n$ big enough and a small constant $\delta$. Therefore,
\begin{equation*}
 Z_n^i \geq \vmin n + \frac{i}{\vmax} (\vmax - \vmin - \varepsilon \delta \vmin \vmax) - n^{3/4}.
\end{equation*}
Using $i \geq (x \vmax + \varepsilon)n$ and choosing $\delta$ small enough finishes the proof.
\end{proof}

\begin{proof}[Proof of Theorem~\ref{theorem_equidistribution}]
We need to show that $\limn \mu_n([0,x]) = \lambda([0,x])$ for every $x \in [0,1]$ almost surely.

Take a realisation of the frog model such that $A_n = G_n$ holds for sufficiently large $n$, that $\lim_{n\to \infty} \frac{M_n}{n}= \vmax$ and $\lim_{n\to \infty} \frac{m_n}{n}= \vmin$, and finally that $A_n \cap \Z^-$ is finite. This happens with probability $1$ as we have seen in Lemma~\ref{lemma_equidistribution_1}, Lemma~\ref{lemma_speed_max}, Theorem~\ref{theorem_speed} and previous discussions about the transience of the frog model.
Now fix $x \in [0,1]$ and $\varepsilon >0$ small.
Lemma~\ref{lemma_equidistribution_2} yields that, for $n$ large enough,
\begin{equation}\label{proof_equidistribution_1} 
\mu_n([0,x]) \geq \frac{1}{|A_n|} |G_n \cap L_n(x,\varepsilon)| = \frac{n}{|A_n|} \cdot \frac{|L_n(x,\varepsilon)|}{n}.
\end{equation}
For the last equation we used that $L_n(x,\varepsilon) \subseteq A_n$ for sufficiently large $n$ as $\limn \frac{M_n}{n} = \vmax$.
The definition of $L_n(x,\varepsilon)$ implies
\begin{equation*}
 |L_n(x,\varepsilon)| \geq 
 \begin{cases}
   2(x\vmax - \varepsilon) n & \text{for $p= \frac12$},\\
   (x\vmax - \varepsilon) n  & \text{for $p> \frac12$}.
 \end{cases}
\end{equation*}
Further, $\limn \frac{n}{|A_n|} = \frac12 \vmax^{-1}$ for $p=\frac12$, respectively $\limn \frac{n}{|A_n|} = \vmax^{-1}$ for $p>\frac12$. Thus, the limit inferior of the last term in \eqref{proof_equidistribution_1} as $n \to \infty$ is bounded from below by $x-\varepsilon \vmax^{-1}$. Since $\varepsilon >0$ was chosen arbitrarily we conclude
\begin{equation*}
\liminf_{n\to\infty}\mu_n([0,x]) \geq x.
\end{equation*}

On the other hand, Lemma~\ref{lemma_equidistribution_2} shows that, for $n$ large enough,
\begin{equation} \label{proof_equidistribution_2}
\mu_n([0,x]) \leq \frac{1}{|A_n|} \bigl|A_n \setminus \bigl(G_n \cap R_n(x,\varepsilon)\bigr)\bigr| 
             = 1 -  \frac{n}{|A_n|} \cdot \frac{|R_n(x,\varepsilon)|}{n} 
\end{equation}
since $A_n = G_n$ and $R_n(x,\varepsilon) \subseteq A_n$ for $n$ big enough. By the definition of $R_n(x,\varepsilon)$ we have
\begin{equation*}
 |R_n(x,\varepsilon)| \geq 
 \begin{cases}
   2\bigl((1-x)\vmax-\varepsilon\bigr)n                      & \text{for $p= \frac12$},\\
   \bigl((1-x)\vmax -2 \varepsilon \bigr)n                   & \text{for $p> \frac12$}.
 \end{cases}
\end{equation*}
Analogous to the above estimation this yields that the limit superior of the right hand side of \eqref{proof_equidistribution_2} is bounded from above by $x+ 2\varepsilon \vmax^{-1}$. As before we get, since $\varepsilon >0$ is arbitrary,
\begin{equation*}
\limsup_{n\to\infty}\mu_n([0,x]) \leq x,
\end{equation*}
which finishes the proof.
\end{proof}


\section{Open Problem}
Simulations suggest that the speed of the maximum is a concave function in the drift parameter $p$. 

\begin{figure}[h]
\centering
\begin{tikzpicture}[scale=1]
 \begin{axis}[
    axis y line =box, 
    axis x line =box, 
    xtickmin=0.5, 
    xlabel=$p$,
    ylabel=$\frac{M_n}{n}$,
    every axis y label/.style={at={(ticklabel cs:0.5)},rotate=0,anchor=near ticklabel},
    xmin=0.5,
    xmax=1,
    ymin=0.48,
    ymax=1,
    enlarge y limits ={value=0.05,upper},
   ]
    
\addplot[only marks, mark size=1pt] coordinates {(0.5,0.5742) (0.51,0.6053) (0.52,0.6391)( 0.53,0.6643)( 0.54,0.6885)( 0.55,0.7178)( 0.56,0.7433)( 0.57,0.7601)( 0.58,0.7804)( 0.59,0.8003)( 0.60,0.8196)( 0.61,0.8384)( 0.62,0.8513)( 0.63,0.8674)( 0.64,0.8837)( 0.65,0.8969)( 0.66,0.9089)( 0.67,0.9187)( 0.68,0.9305)( 0.69,0.9394)( 0.70,0.9473)( 0.71,0.9554)( 0.72,0.9614)( 0.73,0.9684)( 0.74,0.9746)( 0.75,0.9788)( 0.76,0.9844)( 0.77,0.9876)( 0.78,0.9904)( 0.79,0.9931)( 0.80,0.9954)( 0.81,0.9969)( 0.82,0.9977)( 0.83,0.9987)( 0.84,0.9992)( 0.85,0.9996)( 0.86,0.9999)( 0.87,0.9999)( 0.88,1.0000)( 0.89,1.0000)( 0.90,1.0000)( 0.91,1.0000)( 0.92,1.0000)( 0.93,1.0000)( 0.94,1.0000)( 0.95,1.0000)( 0.96,1.0000)( 0.97,1.0000)( 0.98,1.0000)( 0.99,1.0000)( 1,1.0000)};    

\end{axis}
\end{tikzpicture}
\caption{Simulation of $\frac{M_n}{n}$ for $n=100000$}
\end{figure}
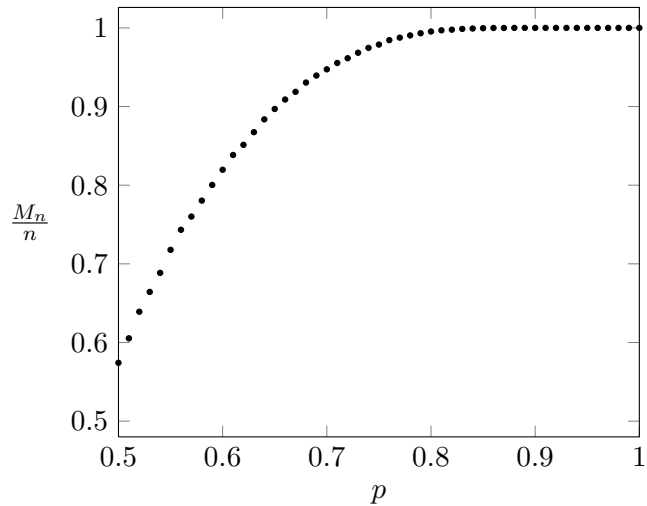

A heuristic argument might be the following.
We expect that the number of frogs in the maximum converges to a stationary distribution $\alpha_p$ for $p<1$. Therefore, the speed of the maximum should equal
\begin{equation*}
 \sum_{k \in \N} \alpha_p(k)\bigl(1-2(1-p)^{k}\bigr).
\end{equation*}
If $\alpha_p$ was independent of $p$, this would be a concave function. However, we believe that the dependence on $p$ does not destroy the concavity.

\begin{con}
The speed of the maximum is a concave function in $p$.  
\end{con}

\paragraph{Acknowledgments} We would like to thank Nina Gantert for introducing us to this problem and for very useful discussions.
Further, we are grateful to Noam Berger for helpful comments and to the unknown referee for several suggestions that helped to improve the readability of the paper.

\bibliographystyle{amsplain}
\bibliography{frogs}

\providecommand{\bysame}{\leavevmode\hbox to3em{\hrulefill}\thinspace}
\providecommand{\MR}{\relax\ifhmode\unskip\space\fi MR }
\providecommand{\MRhref}[2]{%
  \href{http://www.ams.org/mathscinet-getitem?mr=#1}{#2}
}
\providecommand{\href}[2]{#2}
\begin{thebibliography}{10}

\bibitem{AMP02}
O.~S.~M. Alves, F.~P. Machado, and S.~Yu. Popov, \emph{The shape theorem for
  the frog model}, Ann. Appl. Probab. \textbf{12} (2002), no.~2, 533--546.

\bibitem{AMPR01}
O.~S.~M. Alves, F.~P. Machado, S.~Yu. Popov, and K.~Ravishankar, \emph{The
  shape theorem for the frog model with random initial configuration}, Markov
  Process. Related Fields \textbf{7} (2001), no.~4, 525--539.

\bibitem{BR10}
J.~B{\'e}rard and A.~F. Ram{\'{\i}}rez, \emph{Large deviations of the front in
  a one-dimensional model of {$X+Y\to 2X$}}, Ann. Probab. \textbf{38} (2010),
  no.~3, 955--1018.

\bibitem{CQR09}
F.~Comets, J.~Quastel, and A.~F. Ram{\'{\i}}rez, \emph{Fluctuations of the
  front in a one dimensional model of {$X+Y\to2X$}}, Trans. Amer. Math. Soc.
  \textbf{361} (2009), no.~11, 6165--6189.

\bibitem{DP14}
C.~D{\"o}bler and L.~Pfeifroth, \emph{Recurrence for the frog model with drift
  on {$\Z^d$}}, Electron. Commun. Probab. \textbf{19} (2014), no.~79.

\bibitem{GS09}
N.~Gantert and P.~Schmidt, \emph{Recurrence for the frog model with drift on
  {$\Z$}}, Markov Process. Related Fields \textbf{15} (2009), no.~1, 51--58.

\bibitem{GNR15}
A.~P. Ghosh, S.~Noren, and A.~Roitershtein, \emph{On the range of the transient
  frog model on {$\Z$}}, arXiv:1502.02738 [math.PR] (2015).

\bibitem{HJJ14}
C.~Hoffman, T.~Johnson, and M.~Junge, \emph{Recurrence and transience for the
  frog model on trees}, arXiv:1404.6238 [math.PR] (2014).

\bibitem{HJJ15}
\bysame, \emph{From transience to recurrence with poisson tree frogs},
  arXiv:1501.05874 [math.PR] (2015).

\bibitem{KZ15}
E.~Kosygina and P.W. Zerner, \emph{A zero-one law for recurrence and transience
  of frog processes}, arXiv:1508.01953 [math.PR] (2015).

\bibitem{LMP05}
{\'E}.~Lebensztayn, F.~P. Machado, and S.~Yu. Popov, \emph{An improved upper
  bound for the critical probability of the frog model on homogeneous trees},
  J. Stat. Phys. \textbf{119} (2005), no.~1-2, 331--345.

\bibitem{L85}
T.~M. Liggett, \emph{An improved subadditive ergodic theorem}, Ann. Probab.
  \textbf{13} (1985), no.~4, 1279--1285.

\bibitem{P99}
Y.~Peres, \emph{Probability on trees: an introductory climb}, Lectures on
  probability theory and statistics ({S}aint-{F}lour, 1997), Lecture Notes in
  Math., vol. 1717, Springer, Berlin, 1999, pp.~193--280.

\bibitem{P01}
S.~Yu. Popov, \emph{Frogs in random environment}, J. Statist. Phys.
  \textbf{102} (2001), no.~1-2, 191--201.

\bibitem{P03}
\bysame, \emph{Frogs and some other interacting random walks models}, Discrete
  random walks ({P}aris, 2003), Discrete Math. Theor. Comput. Sci. Proc., AC,
  Assoc. Discrete Math. Theor. Comput. Sci., Nancy, 2003, pp.~277--288.

\bibitem{TW99}
A.~Telcs and N.C. Wormald, \emph{Branching and tree indexed random walks on
  fractals}, J. Appl. Probab. \textbf{36} (1999), no.~4, 999--1011.

\end{thebibliography}

\Addresses

\end{document}